\documentclass[journal]{IEEEtran}
\usepackage{cite}
\usepackage{amsmath,amssymb,amsfonts}
\usepackage{amsthm}
\usepackage{graphicx}
\usepackage{hyperref}
\hypersetup{hidelinks}
\usepackage{textcomp}
\def\BibTeX{{\rm B\kern-.05em{\sc i\kern-.025em b}\kern-.08em
    T\kern-.1667em\lower.7ex\hbox{E}\kern-.125emX}}

\newtheorem{theorem}{Theorem}
\newtheorem{proposition}{Proposition}
\newtheorem{corollary}{Corollary}
\newtheorem{lemma}{Lemma}

\newtheorem{remark}{Remark}
\theoremstyle{definition}
\newtheorem{example}{Example}
\theoremstyle{plain}
\DeclareMathOperator{\spec}{spec}
\newcommand{\R}{\mathbb{R}}
\newcommand{\C}{\mathbb{C}}
\newcommand{\rhoSpec}{\rho}

\begin{document}
\title{Stability of MIMO PID With Backward Differences Under Fast Sampling:
An Exact Spectral Criterion}

\author{Zhaobo~Liu, Yanyi~Wu, and Peng~Zhang
\thanks{This work was supported in part by the National Natural Science Foundation of China under
Grant 12401585, in part by the Guangdong Basic and Applied Basic Research Foundation under Grant
2024A1515011542, and in part by the General Program of Shenzhen Natural Science Foundation under
Grant JCYJ20250604181037012.}%
\thanks{Z. Liu and P. Zhang are with the Institute for Advanced Study, Shenzhen University,
Shenzhen, China. Corresponding author: Zhaobo Liu (e-mail: liuzhaobo@szu.edu.cn).}%
\thanks{Y. Wu is with the School of Science and Engineering, The Chinese University of Hong Kong,
Shenzhen, Shenzhen, China (e-mail: 122090587@link.cuhk.edu.cn).}}

\maketitle
\begin{abstract}
Backward differences are a standard digital realization of derivative action in
proportional-integral-derivative control.  This note proves that, for multivariable state-space
plants, this implementation can be unstable no matter how small the sampling period is.  The exact
lifted model reveals fast eigenvalues created by the stored previous output sample.  Away from
boundary spectra, stability is equivalent to ideal loop stability plus Schur stability of the
product of the output, input, and derivative gain matrices.
\end{abstract}


\section{Introduction}
\label{sec:introduction}
\IEEEPARstart{C}{ontrol} engineers often use proportional-integral-derivative (PID) controllers
because they are simple, interpretable, and easy to implement
\cite{AstromHagglund1995,AstromHagglund2001,AstromHagglund2006}.  In a digital controller, however,
the derivative term has to be computed from sampled outputs.  The simplest realization is a backward
difference.  This note asks whether a well-posed and exponentially stable ideal continuous-time PID
loop remains stable under that realization when the sampling period is made arbitrarily small.  This
note studies this question for general multi-input multi-output (MIMO) state-space plants.

Sampled-data and digital control systems have a large stability and design literature, including
input delay and Lyapunov--Krasovskii methods, sampled-data redesign, and continuous-time limit
analyses
\cite{AstromWittenmark1997,ChenFrancis1995,FridmanSeuretRichard2004,Fridman2010,LiuFridman2012,
Mirkin2015,VallarellaHaimovich2022}.  These works provide tools for relating sampled closed loops
to continuous-time designs.  For controllers involving output derivatives and sampled PID
implementations, artificial delay and LMI methods give stability guarantees under fast enough
sampling in settings with relative degree at least two or with second-order plants
\cite{FridmanShaikhet2017,SelivanovFridman2018,SelivanovFridman2019}.
Related root effects also occur in other sampled architectures, such as discrete-time disturbance
observers affected by sampling zeros \cite{ParkLeeJooShim2019}.  The present note studies the
standard MIMO PID realization based on backward differences.

We prove that, for this PID realization, the answer can be negative.  Thus the issue is not merely
that the sampling period was chosen too large.  The previous output sample in the difference
quotient is a controller state.  In the exact lifted model, this state creates fast eigenvalues whose
limits are independent of the ideal continuous-time spectrum.

The mechanism can be seen from the plant
\[
    \dot x(t)=Ax(t)+Bu(t), \qquad y(t)=Cx(t),
\]
with the ideal output PID law
\[
    u(t)=-K_Py(t)-K_Iz(t)-K_D\dot y(t), \qquad \dot z(t)=y(t).
\]
Since \(\dot y=C(Ax+Bu)\), the current input may enter the output derivative directly.  In the
sampled controller, \(\dot y(t)\) is replaced by a difference quotient formed from the current and
previous output samples, so the previous output sample becomes an additional state.  The exact
transition matrix over one sampling period then has slow eigenvalues governed by the ideal PID loop
and fast eigenvalues converging to \(-\spec(CBK_D)\).
Here and below, \(\spec(\cdot)\) denotes the spectrum of a square matrix.

The main contribution is an exact spectral characterization of these fast modes for the standard
backward-difference PID realization.
 To the best of our knowledge, this is
the first such characterization for general MIMO state-space plants.  Away from boundary spectra,
the criterion is simple.  The sampled PID loop is stable for all sampling periods below a positive
threshold if and only if the ideal PID loop is stable and \(CBK_D\) is Schur.  The examples show that
the condition is not a condition on individual gain entries, while \(CB=0\) identifies the structural
case in which the fast spectrum vanishes.

Section~\ref{sec:problem} formulates the ideal and sampled closed loops.
Section~\ref{sec:main-results} gives the main results, from the spectral split to the stability
criterion.
Section~\ref{sec:examples} gives counterexamples and boundary cases.
Section~\ref{sec:discussion} discusses cases with higher relative degree.

\section{Problem Formulation}
\label{sec:problem}

Consider the LTI plant
\begin{equation}
    \dot x(t)=Ax(t)+Bu(t),\qquad y(t)=Cx(t),
    \label{eq:plant}
\end{equation}
where \(x(t)\in\R^n\), \(u(t)\in\R^m\), \(y(t)\in\R^p\), and
\[
    A\in\R^{n\times n},\qquad B\in\R^{n\times m},\qquad C\in\R^{p\times n}
\]
are fixed.  The PID gains satisfy
\[
    K_P,K_I,K_D\in\R^{m\times p}.
\]
Throughout, a matrix is called Hurwitz if all its eigenvalues have negative real parts, and Schur
if all its eigenvalues lie in the open unit disk.
We study regulation about the origin.  Reference signals are omitted because stability is determined
by the homogeneous closed loop.  With negative feedback,
\begin{equation}
    u(t)=-K_Py(t)-K_Iz(t)-K_D\dot y(t),\qquad \dot z(t)=y(t),
    \label{eq:ideal-pid}
\end{equation}
where \(z(t)\in\R^p\).

\subsection{The Ideal Continuous-Time PID Closed Loop}

Substituting \eqref{eq:ideal-pid} into \eqref{eq:plant} gives
\[
    \dot x
    =Ax-BK_PCx-BK_Iz-BK_DC\dot x .
\]
Define
\begin{equation}
    E\triangleq I_n+BK_DC.
    \label{eq:E-def}
\end{equation}
The ideal closed loop is well posed as an ODE for arbitrary initial states if \(E\) is nonsingular.
Such nonsingularity conditions are standard when derivative feedback creates an algebraic loop
\cite{BergerIlchmannTrenn2021}.
Equivalently, by Sylvester's determinant identity,
\[
    \det(I_n+BK_DC)=\det(I_p+CBK_D),
\]
so \(E\) is nonsingular if and only if \(I_p+CBK_D\) is nonsingular.

Under \eqref{eq:E-def}, the ideal augmented state
\[
    s(t)\triangleq\begin{bmatrix}x(t)\\ z(t)\end{bmatrix}\in\R^{n+p}
\]
satisfies \(\dot s(t)=A_{\rm id}s(t)\), where
\begin{equation}
    A_{\rm id}\triangleq
    \begin{bmatrix}
    E^{-1}(A-BK_PC) & -E^{-1}BK_I\\
    C & 0
    \end{bmatrix}.
    \label{eq:Aid}
\end{equation}
The ideal PID loop is exponentially stable if and only if \(A_{\rm id}\) is Hurwitz.

\subsection{The Sampled PID Implementation}

Let \(h>0\), \(t_k=kh\), and
\[
    x_k=x(t_k),\qquad y_k=Cx_k.
\]
The digital integral state \(z_k\in\R^p\) is updated by the rectangular rule
\begin{equation}
    z_{k+1}=z_k+hCx_k .
    \label{eq:z-update}
\end{equation}
The derivative term is implemented by
\[
    \frac{y_k-y_{k-1}}{h}.
\]
The previous output sample is a controller memory.  We write
\[
    r_k\triangleq y_{k-1}\in\R^p,\qquad
    \chi_k\triangleq\begin{bmatrix}x_k\\ z_k\\ r_k\end{bmatrix}\in\R^{n+2p}.
\]
At time \(t_k\), the digital controller computes
\begin{equation}
    u_k\triangleq -K_PCx_k-K_Iz_k-K_D\frac{Cx_k-r_k}{h}.
    \label{eq:sampled-pid}
\end{equation}
The plant input is held constant over the sampling interval, so \(u(t)=u_k\) for all
\(t\in[t_k,t_{k+1})\).
The value \(y_{-1}\) only sets the initial memory \(r_0\).  We allow \(r_0\) to be any vector in
\(\R^p\), rather than requiring \(r_0=Cx_{-1}\) for some earlier plant state \(x_{-1}\).

Define
\[
    \Phi_h\triangleq e^{Ah},\qquad
    \Gamma_h\triangleq\int_0^h e^{A\tau}\,d\tau .
\]
Solving the plant over one sampling interval with the constant input \(u_k\) gives
\begin{equation}
    x_{k+1}=\Phi_hx_k+\Gamma_hBu_k .
    \label{eq:sampled-plant}
\end{equation}
The updates for \(x_k,z_k,r_k\) define a linear autonomous discrete-time closed loop in the state
\(\chi_k\).  Stability of the sampled implementation means exponential stability of this
discrete-time system for every initial triple \((x_0,z_0,r_0)\).  Thus the derivative memory is part
of the sampled controller state, not an external disturbance or approximation error.

\paragraph*{Problem statement}
For the plant \eqref{eq:plant} and the gains \(K_P,K_I,K_D\), assume that the ideal PID loop is
well posed and exponentially stable,
\[
    \det E\ne0,\qquad A_{\rm id}\ \text{is Hurwitz}.
\]
Does there exist \(h^\star>0\) such that the sampled implementation \eqref{eq:z-update},
\eqref{eq:sampled-pid}, \eqref{eq:sampled-plant} is exponentially stable for every
\(h\in(0,h^\star)\)?
The goal is to characterize when this property holds in terms of the derivative channel.

\section{Main Results}
\label{sec:main-results}

This section gives the main results.  First, the exact lifted matrix yields a singular spectral
split into fast and slow modes.  The split is then converted into a stability criterion for small
sampling periods.

\subsection{Exact Lifted Model and Limit Spectrum}
\label{sec:lifted}

The sampled state \(\chi_k\) includes the previous output sample \(r_k\), so the controller based on
the backward difference closes as a finite-dimensional linear system.
Combining \eqref{eq:z-update}, \eqref{eq:sampled-pid}, and \eqref{eq:sampled-plant} gives the exact
lifted discrete-time system
\[
    \chi_{k+1}=\mathcal A_h\chi_k ,
\]
where
\begingroup
\small
\setlength{\arraycolsep}{2.5pt}
\begin{equation}
    \mathcal A_h\triangleq
    \begin{bmatrix}
    \Phi_h-\Gamma_hB(K_P+h^{-1}K_D)C & -\Gamma_hBK_I & h^{-1}\Gamma_hBK_D\\
    hC & I_p & 0\\
    C & 0 & 0
    \end{bmatrix}.
    \label{eq:Ah}
\end{equation}
\endgroup
The sampled implementation is exponentially stable if and only if \(\mathcal A_h\) is Schur.  For
fixed \(h\), the states and input over each sampling interval are bounded linear functions of
\(\chi_k\).  The derivation of \eqref{eq:Ah} uses only the plant solution with a constant input over
one sampling interval and the stated controller updates.

The limit of \(\mathcal A_h\) as \(h\to0^+\) must account for the \(1/h\) factor in the derivative
update.  Here \(\Phi_h\to I_n\) and \(\Gamma_h\to0\), but
\[
    h^{-1}\Gamma_h\to I_n .
\]
Therefore,
\[
    \mathcal A_h\to \mathcal A_0\triangleq
    \begin{bmatrix}
    I_n-BK_DC & 0 & BK_D\\
    0 & I_p & 0\\
    C & 0 & 0
    \end{bmatrix}.
\]

Set, for the rest of the paper,
\begin{equation}
    M_D\triangleq CBK_D\in\R^{p\times p}.
    \label{eq:MD-def}
\end{equation}
The spectrum of the limiting matrix can then be computed explicitly.

\begin{proposition}
\label{prop:limit-spectrum}
With \(M_D\) defined in \eqref{eq:MD-def}, the spectrum satisfies
\begin{equation}
    \spec(\mathcal A_0)=\{1\}^{n+p}\cup\bigl(-\spec(M_D)\bigr),
    \label{eq:limit-spectrum}
\end{equation}
with algebraic multiplicities.  Here \(\{1\}^{n+p}\) denotes the eigenvalue \(1\) repeated
\(n+p\) times.
\end{proposition}

The term \(-\spec(M_D)\) is the fast part of the limiting spectrum.  It comes from the blocks
\(-BK_DC\) and \(BK_D\) in \(\mathcal A_0\), which remain because the backward difference contains
\(1/h\) and stores the previous output sample \(r_k\).  These eigenvalues need not approach \(1\), so
they are not approximations of continuous-time eigenvalues.

\subsection{Spectral Splitting for Small Sampling Periods}
\label{sec:splitting}

Schur stability is determined by whether all eigenvalues lie in the open unit disk.  Thus the fast
limits in Proposition~\ref{prop:limit-spectrum} are decisive whenever they are away from the unit
circle.  More precisely, if a sampled eigenvalue converges to \(-\nu\) with \(|\nu|<1\), then it is
inside the unit disk for all sufficiently small \(h\).  If \(|\nu|>1\), then it is outside the unit
disk for all sufficiently small \(h\).  Hence no first-order expansion is needed for the fast modes
except in boundary cases with \(|\nu|=1\).

The slow eigenvalues are different.  Proposition~\ref{prop:limit-spectrum} says that \(n+p\)
eigenvalues converge to \(1\), which lies on the boundary of the unit disk.  The limit alone
therefore does not indicate on which side of the unit circle these eigenvalues lie.  The relevant
reference is the ideal PID system \(\dot s=A_{\rm id}s\).  If \(\mu\) is an eigenvalue of
\(A_{\rm id}\), then the state transition over a time interval of length \(h\) has eigenvalue
\[
    e^{h\mu}=1+h\mu+O(h^2).
\]
Thus the first-order displacement of the sampled eigenvalues from \(1\) must be identified.  The
theorem below proves that this displacement is governed exactly by the eigenvalues of
\(A_{\rm id}\).  The well-posedness condition also ensures that the slow cluster at \(1\) is
separate from the fast limiting eigenvalues.

\begin{theorem}
\label{thm:spectral-splitting}
Assume that \(E\) is nonsingular.  Let
\(\mu_1,\dots,\mu_{n+p}\) be the eigenvalues of \(A_{\rm id}\), counted with algebraic multiplicity,
and let \(\nu_1,\dots,\nu_p\) be the eigenvalues of \(M_D\), counted with algebraic multiplicity.
Then, for all sufficiently small \(h>0\), the eigenvalues of \(\mathcal A_h\) can be labeled as
\(\lambda_1(h),\dots,\lambda_{n+2p}(h)\) so that, as \(h\to0^+\),
\begin{equation}
    \lambda_j(h)=1+h\mu_j+o(h),
    \qquad j=1,\dots,n+p,
    \label{eq:slow-asymptotics}
\end{equation}
and
\[
    \lambda_{n+p+\ell}(h)\to -\nu_\ell,
    \qquad \ell=1,\dots,p.
\]
\end{theorem}

\begin{remark}
Theorem~\ref{thm:spectral-splitting} gives the main mechanism.  The \(n+p\) slow eigenvalues behave
like the eigenvalues of the ideal PID flow over one sampling interval.  The remaining \(p\)
eigenvalues are different.  They are created by the stored output sample in the backward difference
and converge to \(-\spec(M_D)\), not to \(1\).  Thus fast sampling improves the slow approximation
but does not move these derivative memory modes toward the ideal PID spectrum.
\end{remark}

\begin{corollary}
\label{cor:spectral-radius}
Assume that \(E\) is nonsingular and that \(\rhoSpec(M_D)<1\).  Let
\[
    \alpha_{\rm id}\triangleq
    \max_{\mu\in\spec(A_{\rm id})}\operatorname{Re}\mu .
\]
Then
\[
    \rhoSpec(\mathcal A_h)=1+h\alpha_{\rm id}+o(h),
    \qquad h\to0^+ .
\]
In particular, if \(A_{\rm id}\) is Hurwitz, then \(\alpha_{\rm id}<0\) and
\(\mathcal A_h\) is Schur for all sufficiently small \(h>0\).
\end{corollary}

The proofs of Proposition~\ref{prop:limit-spectrum}, Theorem~\ref{thm:spectral-splitting}, and
Corollary~\ref{cor:spectral-radius} are given in Appendix~\ref{app:spectral-proofs}.

\begin{remark}
Corollary~\ref{cor:spectral-radius} separates stability certification from approximation of the
decay rate.  The assumption \(\rhoSpec(M_D)<1\) keeps the fast eigenvalues in a compact subset of
the open unit disk for all sufficiently small \(h\).  The coefficient of the \(O(h)\) term in
\(\rhoSpec(\mathcal A_h)-1\) is then the spectral abscissa of \(A_{\rm id}\).
\end{remark}

\subsection{Stability and Instability for Small Sampling Periods}
\label{sec:stability}

The spectral split gives two independent sources of stability or instability for small sampling
periods.  The slow eigenvalues enter the unit disk when the ideal PID matrix \(A_{\rm id}\) is
Hurwitz.  The fast eigenvalues enter the unit disk when \(M_D=CBK_D\) is Schur.  If
\(A_{\rm id}\) has an eigenvalue with positive real part, or if \(M_D\) has an eigenvalue outside
the unit disk, the corresponding eigenvalue group leaves the unit disk.  The next theorem records
the sufficiency, the strict instability alternatives, and the equivalence under the two spectral
exclusions in one statement.

\begin{theorem}
\label{thm:stability}
Assume that \(E\) is nonsingular.
The following statements hold.
\begin{enumerate}
\item If \(A_{\rm id}\) is Hurwitz and \(\rhoSpec(M_D)<1\), then there exists \(h^\star>0\) such that
\(\mathcal A_h\) is Schur for every \(h\in(0,h^\star)\).
\item If \(A_{\rm id}\) has an eigenvalue with positive real part, or if \(M_D\) has an eigenvalue
with modulus greater than one, then \(\mathcal A_h\) is not Schur for all sufficiently small
\(h>0\).
\item Suppose, in addition, that \(A_{\rm id}\) has no eigenvalues on the imaginary axis and
\(M_D\) has no eigenvalues on the unit circle.  Then there exists \(h^\star>0\) such that
\(\mathcal A_h\) is Schur for every \(h\in(0,h^\star)\) if and only if \(A_{\rm id}\) is Hurwitz and
\(\rhoSpec(M_D)<1\).
\end{enumerate}
\end{theorem}

\begin{remark}
Theorem~\ref{thm:stability} separates three issues.  The nonsingularity of \(E\) makes the ideal
derivative feedback law an ODE.  The Hurwitz property of \(A_{\rm id}\) places the slow eigenvalues
inside the unit disk for small \(h\).  The condition \(\rhoSpec(M_D)<1\) places the fast eigenvalues
created by the backward difference inside the unit disk.  The spectral exclusions in the last item
are boundary assumptions: if \(A_{\rm id}\) has an eigenvalue on the imaginary axis, or if \(M_D\)
has an eigenvalue on the unit circle, higher-order terms decide stability.  The case of a simple
fast limit on the unit circle is described in Section~\ref{sec:unit-circle-fast}.
\end{remark}

\begin{corollary}
\label{cor:precertification}
Assume that \(E\) is nonsingular, that \(A_{\rm id}\) is Hurwitz, and that
\[
    \spec(M_D)\cap\{z\in\C:|z|=1\}=\emptyset .
\]
Then the sampled PID implementation based on the backward difference is exponentially stable for every
sufficiently small sampling period if and only if
\[
    \rhoSpec(M_D)<1 .
\]
If \(\rhoSpec(M_D)>1\), no sufficiently small sampling period can recover stability of this
implementation.
\end{corollary}

The proofs of Theorem~\ref{thm:stability} and Corollary~\ref{cor:precertification} are given in
Appendix~\ref{app:stability-proofs}.

\begin{remark}
For a chosen PID design, the test for small sampling periods is to form \(A_{\rm id}\) and
\(M_D=CBK_D\), then check that \(A_{\rm id}\) is Hurwitz and \(M_D\) is Schur.  In the SISO case, with
\(g\triangleq CB\) and \(K_D=k_d\), the sampled check is \(|gk_d|<1\), whereas ideal
well-posedness only requires \(1+gk_d\ne0\).  For one prescribed sampling period, stability can also
be verified directly from the lifted matrix by the standard discrete-time Lyapunov test
\cite{Kailath1980}: there exists \(P_h=P_h^\top\succ0\) such that
\[
    \mathcal A_h^\top P_h\mathcal A_h-P_h\prec0 .
\]
\end{remark}

\begin{remark}
The matrix \(M_D=CBK_D\) appears because the derivative is implemented as
\((y_k-y_{k-1})/h\) and the previous output sample is stored as a controller state.  Other digital
PID realizations have different controller states.  For example, derivative filtering, predictors,
or other discrete differentiators change the lifted matrix and may change the fast spectrum.
Therefore \(\rhoSpec(M_D)<1\) is an exact condition for the realization in
\eqref{eq:sampled-pid}, not a universal stability condition for every digital PID implementation.
\end{remark}

\section{Counterexamples and Boundary Cases}
\label{sec:examples}

This section serves three roles.  The scalar integrator gives the smallest counterexample to ideal
PID emulation.  The MIMO example shows that the obstruction is spectral, not componentwise.  The
last example explains why fast limits on the unit circle are excluded from the stability criterion.

\subsection{A Scalar Counterexample}

\begin{example}
Consider the scalar integrator
\[
    \dot x(t)=u(t),\qquad y(t)=x(t).
\]
Let \(K_P=k_p\), \(K_I=k_i\), and \(K_D=k_d\), where \(k_p>0\), \(k_i>0\), and \(k_d>1\).  Then the
ideal continuous-time PID closed loop is well-posed and exponentially stable, but the sampled PID
implementation based on the backward difference is unstable for all sufficiently small \(h>0\).
\end{example}

\begin{proof}
The ideal PID closed loop is
\[
    (1+k_d)\dot x=-k_px-k_iz,\qquad \dot z=x .
\]
Thus
\[
    A_{\rm id}=
    \begin{bmatrix}
    -k_p/(1+k_d)&-k_i/(1+k_d)\\
    1&0
    \end{bmatrix},
\]
which is Hurwitz whenever \(k_p>0\), \(k_i>0\), and \(k_d>-1\).  The ideal continuous-time PID
closed loop is therefore well-posed and exponentially stable for the stated gains.

The sampled implementation based on the backward difference has lifted matrix
\[
    \mathcal A_h=
    \begin{bmatrix}
    1-hk_p-k_d&-hk_i&k_d\\
    h&1&0\\
    1&0&0
    \end{bmatrix}.
\]
The fast mode condition is
\[
    |k_d|<1.
\]
Because \(k_d>1\), Theorem~\ref{thm:stability} implies that the sampled implementation is unstable
for all sufficiently small \(h\).
\end{proof}

Thus no smaller choice of \(h\) can move the limiting fast eigenvalue into the unit disk.  The
failure is caused by the memory introduced by the backward difference, not by plant uncertainty, high
dimension, or a conservative LMI relaxation.

\subsection{A Coupled MIMO Counterexample}

The next example shows a genuinely coupled MIMO effect.  Each entry of the derivative gain is smaller
than one in magnitude, but the derivative channel still has an eigenvalue outside the unit disk.

\begin{example}
Consider
\[
    \dot x(t)=u(t),\qquad y(t)=x(t),
    \qquad x,u,y\in\R^2,
\]
with
\[
    D_0\triangleq
    \begin{bmatrix}
    0.8 & 0.6\\
    0.6 & 0.8
    \end{bmatrix}.
\]
The entries of \(D_0\) all have magnitude less than one, but its eigenvalues are \(1.4\) and \(0.2\).
Set
\[
    K_D=D_0,\qquad K_P=I_2+D_0,\qquad K_I=I_2+D_0 .
\]
Then the ideal continuous-time PID loop is well-posed and exponentially stable, but the sampled PID
implementation based on the backward difference is unstable for all sufficiently small \(h>0\).
\end{example}

\begin{proof}
Here
\[
    E=I_2+D_0
\]
is nonsingular, while \(M_D=CBK_D=D_0\) is not Schur.  Since \(A=0\), \(B=C=I_2\), the ideal PID
matrix becomes
\[
    A_{\rm id}=
    \begin{bmatrix}
    -I_2 & -I_2\\
    I_2 & 0
    \end{bmatrix}.
\]
Its eigenvalues are the roots of \(s^2+s+1=0\), each with multiplicity two, so the ideal PID loop is
stable.  Since a sampled fast eigenvalue converges to \(-1.4\), Theorem~\ref{thm:stability} implies
instability for all sufficiently small \(h\).  The failure is caused by the spectral radius of the
coupled derivative channel, not by a visibly large scalar derivative entry.
\end{proof}

\subsection{Fast Limits on the Unit Circle}
\label{sec:unit-circle-fast}

We now record the first-order drift that decides a simple fast eigenvalue whose limit lies on the
unit circle.

\begin{proposition}
\label{prop:boundary-drift}
Assume that \(E\) is nonsingular.  Let \(\nu\) be a simple eigenvalue of
\(M_D\), and let \(q,\ell\in\C^p\) be right and left eigenvectors satisfying
\[
    M_Dq=\nu q,\qquad \ell^\ast M_D=\nu \ell^\ast,\qquad \ell^\ast q\ne0 .
\]
Set \(\lambda_0=-\nu\).  Then, for all sufficiently small \(h>0\), \(\mathcal A_h\) has a unique
eigenvalue in a fixed small neighborhood of \(\lambda_0\).  Denote this eigenvalue by
\(\lambda_f(h)\).  It satisfies
\[
    \lambda_f(h)=\lambda_0+h\alpha+O(h^2),
\]
where
\begin{equation}
    \alpha=-\frac{\beta_\nu}{(\nu+1)\ell^\ast q},
    \label{eq:boundary-drift}
\end{equation}
with
\[
    \beta_\nu\triangleq
    \nu\,\ell^\ast CBK_Pq
    +\frac{\nu-1}{2}\,\ell^\ast CABK_Dq .
\]
If \(|\nu|=1\), then \(|\lambda_f(h)|>1\) for all sufficiently small \(h>0\) when
\(\operatorname{Re}(\overline{\lambda_0}\alpha)>0\), and \(|\lambda_f(h)|<1\) for all sufficiently
small \(h>0\) when this quantity is negative.  If this quantity is zero, second-order terms are
needed.
\end{proposition}
The proof is given in Appendix~\ref{app:boundary-drift-proof}.

For the scalar integrator, the same drift can be read directly from the characteristic polynomial.
For the scalar lifted matrix,
\begin{equation}
    p_h(\lambda)
    =
    \lambda(\lambda-1)(\lambda-1+hk_p+k_d)
    +h^2k_i\lambda
    -k_d(\lambda-1).
    \label{eq:scalar-char-poly}
\end{equation}
Set \(k_d=1\).  Since \(p_0(\lambda)=(\lambda-1)^2(\lambda+1)\), the fast root \(-1\) is simple and
admits the expansion
\[
    \lambda_f(h)=-1+h\alpha+O(h^2),
\]
Substitution into \eqref{eq:scalar-char-poly} gives
\[
    p_h(\lambda_f(h))
    =
    h(4\alpha+2k_p)+O(h^2).
\]
Hence
\[
    \lambda_f(h)=-1-\frac{k_p}{2}h+O(h^2).
\]
For \(k_p,k_i>0\), this fast root lies outside the unit disk for all small \(h>0\).  This
does not mean that every fast limit on the unit circle is unstable.  Both outcomes can occur.  For
\(\dot x=-2x+u\), \(y=x\), \(k_d=1\), \(k_p=-1\), and \(k_i>0\) small, the ideal matrix
\(A_{\rm id}\) is Hurwitz and Proposition~\ref{prop:boundary-drift} gives
\(\lambda_f(h)=-1+\frac12 h+O(h^2)\), so the fast root initially moves inside.  Thus a fast
limit on the unit circle cannot be classified from the leading spectral split alone.

\section{Relation to Relative-Degree PID Results}
\label{sec:discussion}

This section explains why the fast modes identified in this note do not appear in many sampled PID
results formulated in relative-degree coordinates.  Such analyses often organize the plant around
output derivatives \cite{FridmanShaikhet2017,SelivanovFridman2019}.  When the output has relative
degree at least two, the input does not enter the first output derivative, and hence
\[
    CB=0.
\]

\begin{proposition}
\label{prop:cb-zero}
If \(CB=0\) and \(A_{\rm id}\) is Hurwitz, then there exists \(h^\star>0\) such that the sampled PID
implementation based on the backward difference is exponentially stable for every \(h\in(0,h^\star)\).
\end{proposition}

\begin{proof}
If \(CB=0\), then \(M_D=CBK_D=0\).  Moreover,
\[
    \det(I_n+BK_DC)=\det(I_p+CBK_D)=1,
\]
so the ideal derivative feedback law is well posed.  The result follows from
Theorem~\ref{thm:stability}.
\end{proof}

The proposition shows that the obstruction in this note is absent when the relative degree is at
least two.  The derivative memory still appears in the lifted state, but its limiting fast eigenvalues
are zero.  Thus the additional condition \(\rhoSpec(M_D)<1\) holds automatically.  The difference from
the general state-space case is structural.  When \(CB\ne0\), the current input enters the first output
derivative, and the backward difference leaves a discrete mode that need not approach the slow
continuous-time spectrum.

This distinction also explains the difference between the scalar integrator example and relative
degree two examples.  For the integrator, the output derivative contains the current input, so the
backward difference creates a nonzero fast limit.  For relative degree two, the same calculation gives
\(M_D=0\).  The remaining sampled-data questions then concern delay, robustness, or estimates for
fixed sampling periods, not the fast limiting obstruction studied here.

\begin{remark}
The condition \(\rhoSpec(M_D)<1\) is specific to the realization based on the backward difference
\((y_k-y_{k-1})/h\).  A derivative filter adds controller states and changes the lifted matrix.  The
same method based on lifted states can be applied to filtered derivatives, but the fast matrix must
then be derived from the augmented realization.
\end{remark}

\section{Conclusion}

This note showed that faster sampling alone does not certify stability of the standard sampled PID
implementation based on backward differences.  For general MIMO LTI state-space plants, the exact
lifted model has two groups of eigenvalues.  The slow eigenvalues are governed by the ideal
continuous-time PID loop, while the derivative memory creates fast eigenvalues that converge to
\(-\spec(CBK_D)\) and do not disappear as \(h\) tends to zero.

The resulting criterion for small sampling periods is exact away from boundary spectra.  The sampled
implementation is stable for all sufficiently small sampling periods if and only if the ideal
closed-loop matrix is Hurwitz and \(CBK_D\) is Schur.  Thus a sampled PID design must certify a
property of the implementation itself, not only the stability of the ideal continuous-time law.

This obstruction is absent in relative degree cases with \(CB=0\), which explains why it does not
appear in many formulations based on output derivatives.  In general state-space realizations,
however, the derivative memory is a genuine discrete mode.  Decreasing the sampling period improves
the slow approximation, but it cannot stabilize an unstable fast limit.

\appendices
\section{Proofs of Spectral Results}
\label{app:spectral-proofs}

\begin{lemma}
\label{lem:semisimple-cluster}
Let \(L(h)=L_0+hL_1+O(h^2)\) be a matrix function analytic at \(h=0\).  Suppose that \(\lambda_0\)
is a semisimple eigenvalue of \(L_0\) with algebraic multiplicity \(q\).  Let \(V\) have full column
rank and let \(W\) have full row rank such that
\[
    L_0V=\lambda_0V,\qquad WL_0=\lambda_0W,\qquad WV=I_q .
\]
If \(\eta_1,\dots,\eta_q\) are the eigenvalues of \(WL_1V\), counted with algebraic multiplicity,
then the \(q\) eigenvalues of \(L(h)\) that converge to \(\lambda_0\) can be labeled so that
\[
    \lambda_j(h)=\lambda_0+h\eta_j+o(h),
    \qquad j=1,\dots,q .
\]
\end{lemma}

\begin{proof}
This is the standard first-order perturbation formula for a semisimple matrix eigenvalue cluster;
it is the first step of Kato's reduction process
\cite[Ch.~II, Sec.~2, Th.~2.3 and (2.38)--(2.40), pp.~81--82]{Kato1995}.
\end{proof}

\emph{Proof of Proposition~\ref{prop:limit-spectrum}.}
Permute the state order from \((x,z,r)\) to
\((x,r,z)\).  The matrix \(\mathcal A_0\) becomes block triangular with diagonal blocks
\[
    M_0\triangleq\begin{bmatrix}I_n-BK_DC&BK_D\\ C&0\end{bmatrix},
    \qquad I_p .
\]
Hence
\[
    \det(\lambda I_{n+2p}-\mathcal A_0)
    =(\lambda-1)^p\det(\lambda I_{n+p}-M_0).
\]
For \(\lambda\ne0\), the Schur complement gives
\begin{align*}
    \det(\lambda I_{n+p}-M_0)
    &=
    \lambda^p
    \det\left((\lambda-1)I_n+BK_DC-\lambda^{-1}BK_DC\right)\\
    &=
    \lambda^p(\lambda-1)^n
    \det\left(I_n+\lambda^{-1}BK_DC\right).
\end{align*}
By Sylvester's determinant identity,
\[
    \det\left(I_n+\lambda^{-1}BK_DC\right)
    =
    \det\left(I_p+\lambda^{-1}CBK_D\right).
\]
Therefore
\[
    \det(\lambda I_{n+p}-M_0)
    =
    (\lambda-1)^n\det(\lambda I_p+M_D),
\]
and the identity extends to \(\lambda=0\) by polynomial continuation.  Combining this with the
factor \((\lambda-1)^p\) gives \eqref{eq:limit-spectrum}.

\emph{Semisimple reduction at the eigenvalue \(1\).}
We next identify the first-order perturbation of the eigenvalue cluster at \(1\).  The expansion
follows from
\[
    \Phi_h=I_n+hA+O(h^2),
    \qquad
    \Gamma_h=hI_n+\frac{h^2}{2}A+O(h^3),
\]
with
\[
\mathcal A_1=
\begin{bmatrix}
A-BK_PC-\frac12 ABK_DC & -BK_I & \frac12 ABK_D\\
C & 0 & 0\\
0 & 0 & 0
\end{bmatrix}.
\]
Define
\[
V=
\begin{bmatrix}
I_n&0\\
0&I_p\\
C&0
\end{bmatrix},
\qquad
W=
\begin{bmatrix}
E^{-1}&0&E^{-1}BK_D\\
0&I_p&0
\end{bmatrix}.
\]
A direct calculation gives
\[
    \mathcal A_0V=V,\qquad W\mathcal A_0=W,\qquad WV=I_{n+p}.
\]
By Proposition~\ref{prop:limit-spectrum}, nonsingularity of \(E\) implies that the algebraic
multiplicity of the eigenvalue \(1\) is \(n+p\).  Hence the displayed dual bases show that this
eigenvalue is semisimple.  Furthermore,
\[
    W\mathcal A_1V
    =
    \begin{bmatrix}
    E^{-1}(A-BK_PC)&-E^{-1}BK_I\\
    C&0
    \end{bmatrix}
    =A_{\rm id}.
\]
Thus the first-order reduced matrix for the cluster at \(1\) is exactly \(A_{\rm id}\).

\emph{Proof of Theorem~\ref{thm:spectral-splitting}.}
Proposition~\ref{prop:limit-spectrum} gives \(n+p\)
eigenvalues of the limit matrix at \(1\) and the remaining \(p\) eigenvalues at \(-\spec(M_D)\).
Since \(E\) is nonsingular, \(I_p+M_D\) is nonsingular by Sylvester's determinant identity, and hence
\(1\notin-\spec(M_D)\).  Continuity of polynomial roots separates the two clusters for all
sufficiently small \(h\).  Exactly \(p\) eigenvalues converge, with algebraic multiplicity, to
\(-\spec(M_D)\), and exactly \(n+p\) eigenvalues form the cluster near \(1\).

The entries of \(\mathcal A_h\) are analytic in \(h\) near \(h=0\), with \(h^{-1}\Gamma_h\)
understood through its removable value \(I_n\) at \(h=0\).  The calculation above shows that the
slow cluster is semisimple and that its first-order reduced perturbation is \(A_{\rm id}\).
Lemma~\ref{lem:semisimple-cluster} gives
\[
    \frac{\lambda_j(h)-1}{h}\to\mu_j,
    \qquad j=1,\dots,n+p,
\]
after labeling with algebraic multiplicities.  This proves \eqref{eq:slow-asymptotics}.  The fast
and slow groups account for all eigenvalues.

\emph{Proof of Corollary~\ref{cor:spectral-radius}.}
The condition \(\rhoSpec(M_D)<1\) implies that the fast eigenvalues stay
in a compact subset of the open unit disk for all sufficiently small \(h\).  The slow eigenvalues
satisfy
\[
    |1+h\mu+o(h)|=1+h\operatorname{Re}\mu+o(h)
\]
for the finitely many \(\mu\in\spec(A_{\rm id})\).  Taking the maximum over the slow cluster gives
the stated expansion of \(\rhoSpec(\mathcal A_h)\).

\section{Proofs of Stability Results}
\label{app:stability-proofs}

\emph{Proof of Theorem~\ref{thm:stability}.}
If \(A_{\rm id}\) is Hurwitz, then
\[
    |1+h\mu+o(h)|<1
\]
for each \(\mu\in\spec(A_{\rm id})\) and all small \(h>0\).  If \(\rhoSpec(M_D)<1\), the fast
eigenvalues also remain inside the unit disk.

If \(A_{\rm id}\) has an eigenvalue \(\mu\) with \(\operatorname{Re}\mu>0\), then
\[
    |1+h\mu+o(h)|>1
\]
for small \(h>0\).  If \(M_D\) has an eigenvalue \(\nu\) with \(|\nu|>1\), a fast eigenvalue
converges to \(-\nu\).  Either case prevents Schur stability for all small \(h\).

For the last statement of Theorem~\ref{thm:stability}, the sufficiency follows from the first
paragraph of this proof.  Conversely, assume that \(\mathcal A_h\) is Schur for every sufficiently
small \(h>0\).  Under the stated spectral exclusions, if \(A_{\rm id}\) were not Hurwitz, then it
would have an eigenvalue with positive real part.  If \(\rhoSpec(M_D)<1\) failed, then \(M_D\) would
have an eigenvalue with modulus greater than one.  Each alternative contradicts the instability
conclusion above.

\emph{Proof of Corollary~\ref{cor:precertification}.}
Since \(A_{\rm id}\) is Hurwitz, it has no eigenvalues on the imaginary axis.  The corollary is the
last statement of Theorem~\ref{thm:stability} with the ideal PID stability condition fixed.

\section{Proof of the Boundary Drift Formula}
\label{app:boundary-drift-proof}

Since \(\nu\) is a simple eigenvalue of \(M_D\), Proposition~\ref{prop:limit-spectrum} implies that
\(\lambda_0=-\nu\) is a simple eigenvalue of \(\mathcal A_0\).  Therefore, by simple-eigenvalue
perturbation, there is a unique eigenvalue branch \(\lambda_f(h)\) of \(\mathcal A_h\) near
\(\lambda_0\) for all sufficiently small \(h>0\).
The expansion \(\mathcal A_h=\mathcal A_0+h\mathcal A_1+O(h^2)\) follows from
\(\Phi_h=I_n+hA+O(h^2)\) and \(\Gamma_h=hI_n+\frac{h^2}{2}A+O(h^3)\), with
\[
\mathcal A_1=
\begin{bmatrix}
A-BK_PC-\frac12 ABK_DC & -BK_I & \frac12 ABK_D\\
C & 0 & 0\\
0 & 0 & 0
\end{bmatrix}.
\]
A right eigenvector of \(\mathcal A_0\) for \(\lambda_0=-\nu\) is
\[
    v=\begin{bmatrix}-BK_Dq\\0\\q\end{bmatrix},
\]
and a left eigenvector is
\[
    w^\ast=\begin{bmatrix}\ell^\ast C&0&-\ell^\ast\end{bmatrix}.
\]
Because \(E\) is nonsingular, \(\nu\ne -1\), and
\[
    w^\ast v=-(\nu+1)\ell^\ast q\ne0 .
\]
The standard first-order perturbation formula for a simple eigenvalue
\cite{StewartSun1990,Kato1995} gives
\[
    \alpha=\frac{w^\ast \mathcal A_1v}{w^\ast v}.
\]
Moreover,
\[
    w^\ast \mathcal A_1v
    =
    \nu\,\ell^\ast CBK_Pq+\frac{\nu-1}{2}\,\ell^\ast CABK_Dq,
\]
which yields \eqref{eq:boundary-drift}.  If \(|\nu|=1\), the statement about the unit disk follows
from
\[
    |\lambda_0+h\alpha|^2
    =
    1+2h\,\operatorname{Re}(\overline{\lambda_0}\alpha)+O(h^2).
\]

\end{document}